\documentclass[graybox]{svmult}

\usepackage{mathptmx}       
\usepackage{helvet}         
\usepackage{courier}        
\usepackage{type1cm}        
\usepackage{makeidx}         
\usepackage{graphicx,xcolor}        
\usepackage{multicol}        
\usepackage[bottom]{footmisc}

\usepackage{amssymb,nicefrac,bm,upgreek,verbatim,mathtools,comment,cite}
\RequirePackage[mathscr]{eucal}
\usepackage{dsfont}
\usepackage[final]{hyperref}
\hypersetup{
 colorlinks=true,
 citecolor=blue,
 linkcolor=blue,
 frenchlinks=false,
 pdfborder={0 0 0},
 naturalnames=false,
 hypertexnames=false,
 breaklinks}

\newcommand{\Ind}{\mathds{1}} 
\newcommand{\process}[1]{{\{#1_t\}_{t\ge0}}}

\newcommand{\Exp}{{\mathbb{E}}} 
\newcommand{\Prob}{{\mathbb{P}}} 
\newcommand{\RR}{{\mathbb R}} 
\newcommand{\ZZ}{{\mathbb Z}} 
\newcommand{\Rm}{{{\mathbb R}^m}}

\newcommand{\NN}{{\mathbb N}} 
\newcommand{\df}{\coloneqq} 

\newcommand{\sB}{{\mathscr{B}}} 
\newcommand{\cB}{{\mathcal{B}}} 
\newcommand{\Usm}{\mathfrak{U}_{\mathrm{sm}}}  
\newcommand{\tUsm}{\widetilde{\mathfrak{U}}_{\mathrm{sm}}} 
\newcommand{\Cc}{\mathnormal{C}}
    
\newcommand{\Ag}{{\mathscr{A}}}
\newcommand{\cA}{{\mathcal{A}}}
\newcommand{\Lg}{{\mathscr{L}}}
\newcommand{\fI}{{\mathfrak{I}}}

\newcommand{\cI}{{\mathscr{I}}}   
\newcommand{\Lyap}{{\mathscr{V}}}
\newcommand{\cP}{{\mathscr{P}}} 
\newcommand{\cK}{{\mathscr{K}}} 
\newcommand{\tv}{{\rule[-.45\baselineskip]{0pt}{\baselineskip}\mathsf{TV}}}

\newcommand{\abs}[1]{\lvert#1\rvert}
\newcommand{\norm}[1]{\lVert#1\rVert}
\newcommand{\babs}[1]{\bigl\lvert#1\bigr\rvert}
\newcommand{\Babs}[1]{\Bigl\lvert#1\Bigr\rvert}
\newcommand{\babss}[1]{\biggl\lvert#1\biggr\rvert}

\newcommand{\bnorm}[1]{\bigl\lVert#1\bigr\rVert}

\DeclareMathOperator*{\diag}{diag}
\DeclareMathOperator*{\Argmin}{Arg\,min}

\DeclareMathOperator{\trace}{trace}

\begin{document}

\title*{Uniform polynomial rates of convergence for a class of
L\'evy--driven
controlled SDEs
arising in multiclass  many-server queues}
\titlerunning{Uniform polynomial rates of convergence for a class of 
L\'evy--driven SDEs} 
\author{Ari Arapostathis, Hassan Hmedi, Guodong Pang, and Nikola Sandri{\'c}}
\institute{Ari Arapostathis \at Department of ECE,
The University of Texas at Austin, EER~7.824, Austin, TX~~78712, \email{ari@ece.utexas.edu}
\and Hassan Hmedi \at Department of ECE,
The University of Texas at Austin,
EER~7.834,
Austin, TX~~78712, \email{hmedi@utexas.edu}
\and Guodong Pang \at 
 The Harold and Inge Marcus Dept. of Industrial and Manufacturing Eng.,
College of Engineering, Pennsylvania State University, University Park, PA~~16802, 
\email{gup3@psu.edu} 
\and Nikola Sandri{\'c} \at  Department of Mathematics, University of Zagreb,
Bijeni\v{c}ka cesta 30,
10000 Zagreb, Croatia,   \email{nsandric@math.hr} 
}

\maketitle

\abstract{
We study the ergodic properties of a class of 
controlled stochastic differential equations (SDEs) driven by
$\alpha$-stable processes which arise as the
limiting equations of multiclass
queueing models in the Halfin--Whitt regime that have heavy--tailed arrival processes. 
When the safety staffing parameter is positive, we show that the SDEs are uniformly
ergodic and enjoy a polynomial rate of convergence to the invariant
probability measure in total variation, which
is uniform over all stationary Markov controls resulting in a locally Lipschitz
continuous drift. 
We also derive a matching lower bound on the rate of convergence
(under no abandonment).
On the other hand, when all abandonment rates are positive, we show that the
SDEs are exponentially ergodic uniformly over the above-mentioned class of controls. 
Analogous results are obtained for L\'evy--driven SDEs arising from
multiclass many-server queues under
 asymptotically negligible service interruptions.  For these equations,
we show that the aforementioned ergodic properties are uniform over all
stationary Markov controls.
We also extend a key functional central limit theorem concerning diffusion
approximations so as to make it applicable
to the models studied here. 
}

\section{Introduction}

L{\'e}vy--driven controlled stochastic differential equations (SDEs) arise as
scaling limits for multiclass many-server
queues with heavy-tailed arrival processes and/or with asymptotically negligible
service interruptions; see \cite{PW09, PW10, APS19}. 
In these equations, the control appears only in the drift and corresponds to a
work-conserving scheduling policy in multiclass many-server queues, that is, the
allocation of the available service capacity to each class under a non-idling condition
(no server idles whenever there are jobs in queue).
For the limiting process, we focus on  stationary Markov controls, namely
time-homogeneous functions of the process.
When the arrival process of each class is heavy-tailed
(for example, with regularly varying interarrival times),
the L{\'e}vy process driving the SDE is a multidimensional anisotropic
$\alpha$-stable process, $\alpha \in (1,2)$.
When the system is subject to service interruptions (in an alternating renewal
environment affecting the service processes only), the L{\'e}vy process is
a combination of  either
a Brownian motion, or an anisotropic $\alpha$-stable process,
$\alpha \in (1,2)$,  and an independent compound Poisson process.

Ergodic properties of these controlled SDEs are of great interest since they
help to understand
the performance of the queueing systems. 
In \cite{APS19}, the ergodic properties of the SDEs under constant controls
are thoroughly studied. It is shown that when the safety staffing is positive,
the SDEs have a polynomial rate of convergence to stationarity in total variation,
while when the abandonment rates are positive, the rate of convergence is exponential.
However,  the technique developed in \cite{APS19} does not equip us to investigate
the ergodic properties of these SDEs beyond the constant controls,
since the Lyapunov functions employed are modifications of the common quadratic
functions that have been developed for piecewise linear diffusions \cite{DG13}. 

It was recently shown in \cite{GS12} that the Markovian multiclass many--server queues
with positive safety staffing
in the Halfin--Whitt regime are stable under any work-conserving scheduling policies. 
Motivated by this significant result, Arapostathis et al. (2018) \cite{AHP18} have
developed a unified approach via a Lyapunov function method which establishes
Foster-Lyapunov equations which are uniform
under stationary Markov controls for the
limiting diffusion and the prelimit diffusion-scaled queueing processes simultaneously.
It is shown that the limiting diffusion is uniformly exponentially ergodic under
any stationary Markov control.

In this paper we adopt and extend the approach in \cite{AHP18} to establish
uniform ergodic properties for L{\'e}vy-driven SDEs.
As done in \cite{APS19},
we distinguish two cases: (i) positive safety staffing,
and (ii) positive abandonment rates. 
We focus primarily on the first case, which exhibits
ergodicity at a polynomial rate, a result which is somewhat surprising.
The second case always results in uniform exponential ergodicity.  
By employing a polynomial Lyapunov function instead of the exponential function used in
 \cite{AHP18}, we first establish an upper bound on the rate of convergence
 which is polynomial.
The drift inequalities carry over with slight modifications
from \cite{AHP18}, while the needed properties of the non-local part
of the generator
are borrowed from \cite{ABC-16}. 
As in \cite{APS19}, we use the technique in \cite{Hairer-16} to establish
a lower bound on the rate of convergence, which actually matches the upper bound.  
As a result, we establish that with positive safety staffing, the rate of convergence
to stationarity in total variation is polynomial with a rate that is
uniform over the family of Markov controls
which result in a locally Lipschitz continuous drift.

When the SDE is driven by an $\alpha$--stable process (isotropic or anisotropic),
in order for the process to be open--set irreducible and aperiodic,
it suffices to require that the controls are stationary Markov and
the drift is locally Lipschitz continuous. 
However, the existing proof of the convergence of the scaled queueing processes
of the multiclass many--server queues with heavy--tailed arrivals to this limit process,
assumes that the drift is Lipschitz continuous \cite{PW10}.
In this paper, we extend this result on the continuity of the integral mapping
(Theorem 1.1 in \cite{PW10}) to drifts that are
locally Lipschitz continuous with at most linear growth (see Lemma~\ref{L4}).
Applying this, we also present an extended  functional central limit theorem (FCLT) for
 multiclass many-server queues with heavy-tailed arrival processes (see Theorem~\ref{T6}).

On the other hand, when the L\'evy process consists of a Brownian motion
and a compound Poisson process, which arises in the multiclass many--server queues
with asymptotically negligible interruptions under the $\sqrt{n}$ scaling,
the SDE has a unique strong solution that is open--set irreducible and aperiodic
under any stationary Markov control.
To study uniform ergodic properties, we also need to account for the second order
derivatives in the infinitesimal generator.
For this reason we modify the Lyapunov function
with suitable titling on the positive and negative half state spaces.
We also discuss the model with a L\'evy process consisting
of a $\alpha$-stable process and a compound Poisson process. 

\subsection{Organization of the paper}
In Section~\ref{S2}, we present a class of SDEs driven by an $\alpha$--stable process,
whose ergodic properties are studied in Section~\ref{S3}. 
In Section~\ref{S4}, we study the ergodic properties of
L{\'e}vy--driven SDEs arising from the multiclass queueing models
with service interruptions. 
In Section~\ref{S5}, we provide a description of the 
multiclass many--server queues with heavy-tailed arrival processes, 
and establish the continuity of the integral mapping with 
a locally Lipschitz continuous function that has at most linear growth,
as well as the associated FCLT.

\subsection{Notation}

We summarize some notation used throughout the paper. 
We use $\Rm$ (and $\mathbb{R}^m_+$), $m\ge 1$,
to denote real-valued $m$-dimensional (nonnegative) vectors, and write $\RR$ for $m=1$.
For $x, y\in \RR$, we write $x \vee y = \max\{x,y\}$,
$x\wedge y = \min\{x,y\}$, $x^+ = \max\{x, 0\}$ and $x^- = \max\{-x,0\}$.
For a set $A\subseteq\Rm$, we use $A^c$, $\partial A$, and $\Ind_{A}$
to denote the complement, the boundary, and the indicator function of $A$, respectively.
A ball of radius $r>0$ in $\Rm$ around a point $x$ is denoted by $\sB_{r}(x)$,
or simply as $\sB_{r}$ if $x=0$. We also let $\sB \equiv \sB_{1}$. 
The Euclidean norm on $\Rm$ is denoted by $\abs{\,\cdot\,}$,
and $\langle \cdot\,,\,\cdot\rangle$ stands for the inner product.
For $x\in\Rm$, we let $\norm{x}^{}_1\df \sum_i \abs{x_i}$,
and we use $x'$ to denote the transpose of $x$.
We use the symbol $e$ to denote the vector whose elements are all equal to $1$,
and $e_i$ for the vector whose $i^{\text{ th}}$ element is equal to $1$ and
the rest are equal to $0$.

We let $\cB(\Rm)$, $\cB_b(\Rm)$, and $\cP(\Rm)$ denote the classes of Borel
measurable functions, bounded Borel measurable functions,
and Borel probability measures on $\Rm$, respectively.
By $\cP_p(\Rm)$, $p>0$, we denote the subset of $\cP(\Rm)$
containing all probability measures $\uppi(\D x)$ with the property that
$\int_{\Rm}\abs{x}^p\uppi(\D{x})<\infty$.
For a finite signed measure $\nu$ on $\Rm$, and a Borel measurable
$f\colon\Rm\to[1,\infty)$,
$\norm{\nu}_{f} \df \sup_{\abs{g}\le f}\,\int_{\Rm} \abs{g(x)}\,\nu(\D{x})$,
where the supremum is over all Borel measurable functions $g$ satisfying this inequality.

\section{The model}\label{S2}
We consider an $m$-dimensional stochastic differential equation (SDE) of the form
\begin{equation}\label{E-sde}
\D X_t\,=\,b(X_t,U_t)\,\D{t} + \D \Hat{A}_t,\qquad X_0=x\in\Rm\,.
\end{equation}
All random processes in \eqref{E-sde} live in a complete
probability space $(\varOmega,\mathcal{F},\Prob)$.
We have the following structural hypotheses.
\begin{enumerate}
\item [(A1)]
The control process $\process{U}$ lives in the $(m-1)$-simplex
\begin{equation*}
\Delta \,\df\, \{ u\in\Rm \,\colon u\ge0\,,\ \langle e,u\rangle = 1\}\,,
\end{equation*}
and the drift $b\colon\Rm\times \Delta\to\Rm$
 is given by
\begin{equation}\label{E-drift}
\begin{aligned}
b(x,u) &\,=\,
\ell-M\bigl(x-\langle e,x\rangle^+u\bigr)-\langle e,x\rangle^+\Gamma u\\[5pt]
&\,=\, \begin{cases}
\ell - \bigl(M +(\Gamma-M)u e'\bigr) x\,, & \langle e,x\rangle>0\,,\\[2pt]
\ell - Mx\,, &\langle e,x\rangle\le0\,,
\end{cases}
\end{aligned}
\end{equation}
where  $\ell \in \Rm$,
$M =\diag(\mu_1,\dotsc,\mu_m)$ with $\mu_i>0$, and
$\Gamma=\diag(\gamma_1,\dotsc,\gamma_m)$ with $\gamma_i\in\RR_{+}$, $i=1,\dotsc,m$.

\item [(A2)]
The process
$\process{\Hat{A}}$
is an anisotropic L\'evy process with independent symmetric one-dimensional
$\alpha$-stable components for $\alpha \in (1,2)$. 
\end{enumerate}

Define
\begin{equation*}
\cK_+ \,\df\, \bigl\{ x\in\Rm\,\colon \langle e,x\rangle
> 0\bigr\}\,,\quad\text{and\ \ }
\cK_- \,\df\, \bigl\{ x\in\Rm\,\colon \langle e,x\rangle
\le 0\bigr\}\,.
\end{equation*}
A control $U_t$ is called stationary Markov, if it takes
the form $U_t=v(X_t)$ for a Borel measurable function $v\colon\cK_+\to\Delta$. 
We let $\Usm$ denote the class of stationary Markov controls, and $\tUsm$ its
subset consisting of those controls under which
\begin{equation*}
b_v(x)\,\df\,b\bigl(x,v(x)\bigr)
\end{equation*}
is locally Lipschitz continuous.
These controls can be identified with the function $v$.
Note that if $v\colon\cK_+\to\Delta$ is Lipschitz continuous when
restricted to any set $\cK_+\cap\sB_R$, $R>0$, then $v\in\tUsm$,
but this property is not necessary for
membership in $\tUsm$.

Clearly, for any $v\in\Usm$,
the drift $b_v(x)$  has at most linear growth.
Therefore, if $v\in\tUsm$, then using
\cite[Theorem~3.1, and Propositions~4.2 and~4.3]{Albeverio-2010},
one can conclude that the SDE \eqref{E-sde} admits a unique nonexplosive strong
solution $\process{X}$ which is a strong Markov process and it satisfies the
$C_b$-Feller property.
In addition, in the same reference, it is shown that the infinitesimal generator
$(\Ag^v,\mathcal{D}_{\Ag^v})$ of $\process{X}$
(with respect to the Banach space $(\mathcal{B}_b(\RR^{m}),\norm{\,\cdot\,}_\infty)$)
satisfies $C_c^{2}(\Rm)\subseteq\mathcal{D}_{\Ag^v}$ and 
\begin{equation}\label{E-Ag}
\Ag^v\bigr|_{C_c^{2}(\Rm)}f(x) \,\df\,
\bigl\langle b_v(x),\nabla f(x)\bigr\rangle
+ \fI_\alpha f(x)\,,
\end{equation}
where
\begin{equation*}
\fI_\alpha f(x)\,\df\,\sum_{i=1}^d\int_{\RR_*}
\mathfrak{d} f(x; y_i e_i)\,\frac{\xi_i\,\D{y}_i}{\abs{y_i}^{1+\alpha}}\,,
\end{equation*}
for some positive constants $\xi_1,\dotsc,\xi_m$,
and
\begin{equation}\label{frakd}
\mathfrak{d} f(x;y) \,\df\,
f(x+y)-f(x)-\langle y,\nabla f(x)\rangle\,,\qquad f\in C^1(\Rm)\,.
\end{equation}
Here, $\mathcal{D}_{\Ag^v}$
and $C_c^{2}(\Rm)$
denote the domain of $\Ag^v$ 
and the space of twice continuously differentiable functions with compact support,
respectively. 

We let  $\Prob^v_x$ and $\Exp^v_x$ denote the probability measure and
expectation operator on the canonical space of the solution of \eqref{E-sde} under
$v\in\tUsm$  and starting at $x$.
Also, $P^v_t(x,\D y)$ denotes its transition probability.
From  the proof of Theorem~3.1\,(iv) in \cite{APS19} we have the following result.

\begin{theorem}\label{T1}
Under any $v\in\tUsm$, $P^v_t(x,B)>0$ for all $t>0$, $x\in\Rm$ and $B\in\cB(\Rm)$
with positive Lebesgue measure. In particular, under any $v\in\tUsm$,
the process $\process{X}$ is open--set irreducible and aperiodic in the
sense of \cite{Meyn-Tweedie}.
\end{theorem}

\begin{remark}
As far as the results in this paper are concerned we can replace
the anisotropic non-local operator $\fI_\alpha$ with
the isotropic operator
\begin{equation*}
\int_{\RR_*}
\mathfrak{d} f(x; y)\,\frac{\D{y}}{\abs{y}^{m+\alpha}}\,,
\end{equation*}
as done in \cite{APS19}.
\end{remark}

We also define
\begin{equation*}
 \Ag^u f(x) \,\df\,
\bigl\langle b(x,u),\nabla f(x)\bigr\rangle + \fI_\alpha f(x)\,,\quad u\in\Delta\,.
\end{equation*}

In the next section we study the ergodic properties of $\process{X}$. 
To facilitate the analysis, we define the \emph{spare capacity},
or \emph{safety staffing},
$\beta$ as
\begin{equation}\label{E-varrho}
\beta \,\df\, - \langle\,e,M^{-1}\ell\,\rangle\,.
\end{equation}
Note that if we let $\zeta =\frac{\beta}{m} e + M^{-1}\ell$,
with $\beta$ as in \eqref{E-varrho}, then a mere translation of the origin of the form
$\Tilde X_t = X_t - \zeta$
results in an SDE of the same form, with the only difference
that the constant term $\ell$ in the drift equals $- \frac{\beta}{m} Me$.
Since translating the origin does not alter the ergodic properties of the process,
without loss of generality, we assume throughout the paper that the drift
in \eqref{E-drift} has the form
\begin{equation}\label{E-drift2}
b(x,u) \,=\, -\frac{\beta}{m}Me
-M(x-\langle e,x\rangle^+u)-\langle e,x\rangle^+\Gamma u\,.
\end{equation}

\section{Uniform ergodic properties}\label{S3}

We recall some important definitions used in
\cite[Section~2.3]{AHP18}. 

\begin{definition}\label{D1}
We fix some convex function $\psi\in\Cc^2(\RR)$ with the property that
$\psi(t)$ is constant for $t\le-1$, and $\psi(t)=t$ for $t\ge0$.
The particular form of this function is not important.
But to aid some calculations we fix this function as
\begin{equation*}
\psi(t) \df\,
\begin{cases}
-\frac{1}{2}, & t\le-1\,,\\[3pt]
(t+1)^3 -\frac{1}{2}(t+1)^4-\frac{1}{2} & t \in [-1,0]\,,\\[3pt]
t & t\ge 0\,.
\end{cases}
\end{equation*}
Let $\cI = \{1,\dotsc,m\}$.
With $\delta$ and $p$ positive constants, we define
\begin{equation*}
 \Psi(x) \,\df\, \sum_{i\in\cI} \frac{\psi(x_i)}{\mu_i}\,,\quad\text{and}\quad
V_{p}(x) \,\df\, \biggl(\delta \Psi(-x)+ \Psi(x)
+ \frac{m}{\min_{i\in\cI}\mu_i}\biggr)^p\,.
\end{equation*}
\end{definition}
Note that the term inside the parenthesis in the definition of $V_p$, or
in other words $V_1$, is bounded away from $0$ uniformly in $\delta\in(0,1]$.
The function $V_p$ also depends on the parameter $\delta$
which is suppressed in the notation.

For $x\in\Rm$ we let $x^\pm\df \bigl(x_1^\pm,\dotsc,x_m^\pm\bigr)$.
The results which follows is a corollary of Lemma~2.1 in \cite{AHP18}, but
we sketch the proof for completeness.

\begin{lemma}\label{L1}
Assume $\beta>0$, and let $\delta\in(0,1]$ satisfy
\begin{equation}\label{EL1A}
\Bigl(\max_{i\in\cI} \tfrac{\gamma_i}{\mu_i}-1\Bigr)^+\delta \,\le\, 1\,.
\end{equation}
Then, the function $V_p$ in Definition~\ref{D1} satisfies, for any $p>1$
and for all $u\in\Delta$,
\begin{align}
\bigl\langle b(x,u),\nabla V_p (x)\bigr\rangle &\,\le\,
p\Bigl(\delta\beta +
\frac{m}{2}(1+\delta) - \delta \norm{x}^{}_1\Bigr) V_{p-1}(x)
\quad\forall\,x\in\cK_-\,, \label{EL1B}\\
\bigl\langle b(x,u),\nabla V_p(x)\bigr\rangle &\,\le\,
- p\Bigl(\frac{\beta}{m}-\delta\beta-\delta\frac{m}{2}
+\delta\norm{x^-}{}_1\Bigr)\,V_{p-1}(x)
\quad\forall\,x\in\cK_+\,.\label{EL1C}
\end{align}
\end{lemma}

\begin{proof}
We have
\begin{equation}\label{PL1A}
\begin{aligned}
\bigl\langle b(x,u),\nabla \Psi(x)\bigr\rangle &\,=\,
-\frac{\beta}{m} \sum_{i\in\cI} \psi'(x_i)
- \sum_{i\in\cI} \psi'(x_i)\bigl(x_i-\langle e,x\rangle^+ u_i\bigr)\\
&\mspace{250mu}- \langle e,x\rangle^+ \sum_{i\in\cI}
\psi'(x_i)\tfrac{\gamma_i}{\mu_i} u_i\,,
\end{aligned}
\end{equation}
and
\begin{equation}\label{PL1B}
\begin{aligned}
\bigl\langle b(x,u),\nabla \Psi(-x)\bigr\rangle &\,=\,
\frac{\beta}{m} \sum_{i\in\cI} \psi'(-x_i) + \sum_{i\in\cI} \psi'(-x_i)x_i\\
&\mspace{80mu}- \langle e,x\rangle^+ \sum_{i\in\cI} \psi'(-x_i)
\bigl(1-\tfrac{\gamma_i}{\mu_i}\bigr)^+ u_i \\
&\mspace{160mu}+ \langle e,x\rangle^+ \sum_{i\in\cI} \psi'(-x_i)
\bigl(\tfrac{\gamma_i}{\mu_i}-1\bigr)^+ u_i \,.
\end{aligned}
\end{equation}

It is easy to verify that $\psi'(-\nicefrac{1}{2}) = \nicefrac{1}{2}$,
from which we obtain
\begin{equation}\label{PL1C}
\sum_{i\in\cI} \psi'(x_i) x_i \,\ge\, \norm{x^+}^{}_1 - \frac{m}{2}\,,
\quad\text{and\ \ }
-\sum_{i\in\cI} \psi'(-x_i) x_i \,\ge\, \norm{x^-}^{}_1 - \frac{m}{2}\,.
\end{equation}
Therefore, \eqref{EL1B} follows by using \eqref{PL1C} in
\eqref{PL1A}--\eqref{PL1B}.

We next turn to the proof of \eqref{EL1C}.
If $\gamma_i\le\mu_i$ for all $i\in\cI$, then the proof
is simple.
This is because the inequality
$\sum_{i\in\cI} \psi'(x_i) x_i \ge \langle e,x\rangle$ 
and  the fact that $\norm{\psi'}_\infty\le1$
implies that
\begin{equation*}
\sum_{i\in\cI} \psi'(x_i)\bigl(x_i-\langle e,x\rangle^+ u_i\bigr)\,\ge\,0
\quad\text{for\ } x\in\cK_+\,,
\end{equation*}
which together with \eqref{PL1A} shows that
\begin{equation}\label{PL1D}
\bigl\langle b(x,u), \nabla \Psi(x) \bigr\rangle \,\le\,
-\frac{\beta}{m} \sum_{i\in\cI} \psi'(x_i)
\,\le\, -\frac{\beta}{m}\quad\text{on\ }\cK_+\,.
\end{equation}
On the other hand, by \eqref{PL1B} and \eqref{PL1C} we obtain
\begin{equation}\label{PL1E}
\begin{aligned}
\delta\bigl\langle b(x,u), \nabla \Psi(-x) \bigr\rangle &\,\le\,
\delta \frac{\beta}{m} \sum_{i\in\cI} \psi'(-x_i)
+ \delta \sum_{i\in\cI} \psi'(-x_i)x_i\\
&\,\le\, \delta\beta + \delta \frac{m}{2} - \delta\norm{x^-}{}_1
\quad\text{on\ }\Rm\,.
\end{aligned}
\end{equation}
Therefore, when $\gamma_i\le\mu_i$ for all $i\in\cI$,
\eqref{EL1C} follows by adding \eqref{PL1D} and \eqref{PL1E}.

Without assuming that $\gamma_i\le\mu_i$, a
careful comparison of the terms in \eqref{PL1A}--\eqref{PL1B}, shows that
(see \cite[Lemma~2.1]{AHP18})
\begin{equation}\label{PL1F}
\begin{aligned}
&\delta \langle e,x\rangle^+ \sum_{i\in\cI} \psi'(-x_i)
\bigl(\tfrac{\gamma_i}{\mu_i}-1\bigr)^+ u_i
- \sum_{i\in\cI} \psi'(x_i)\bigl(x_i-\langle e,x\rangle^+ u_i\bigr)\\
&\mspace{150mu}- \langle e,x\rangle^+ \sum_{i\in\cI}
\psi'(x_i)\tfrac{\gamma_i}{\mu_i} u_i\,\le\,0\quad \forall\,(x,u)\in\cK_+\times\Delta\,.
\end{aligned}
\end{equation}
Thus \eqref{EL1C} follows by using \eqref{PL1D}--\eqref{PL1F} in
\eqref{PL1A}--\eqref{PL1B}.
This completes the proof.
\qed\end{proof}

On the other hand, when $\Gamma>0$,
the proof of  \cite[Theorem~2.2]{AHP18} implies the following.

\begin{lemma}\label{L2}
Assume that $\Gamma>0$.
Then there exists a  positive constant $\delta$ such that for any $p>1$,
\begin{equation*}
\bigl\langle b(x,u),\nabla V_p(x)\bigr\rangle \,\le\,
c_0 - c_1 V_p(x) \quad\forall\,(x,u)\in\RR^m\times\Delta\,,
\end{equation*}
for some positive constants $c_0$ and $c_1$ depending only on $\delta$.
\end{lemma}

Another result that we borrow is Proposition~5.1 in \cite{ABC-16}, whose proof implies
the following.

\begin{lemma}\label{L3}
The map $x\mapsto \abs{x}^{\alpha-p}\,\fI_\alpha V_p(x)$
is bounded on $\Rm$ for any $p\in(0,\alpha)$.
\end{lemma}

Theorems~\ref{T2} and \ref{T3} that follow
establish ergodic properties which are uniform
over controls in $\tUsm$
in the case of positive safety staffing
and positive abandonment rates, respectively.

\begin{theorem}\label{T2}
Assume $\beta>0$. In addition to \eqref{EL1A}, let
\begin{equation}\label{ET2A}
\delta\,<\, \frac{\beta}{2m(2\beta + m)}\,.
\end{equation}
We have the following.
\begin{enumerate}
\item[\upshape{(}a\upshape{)}]
For any $p\in(1,\alpha)$,
the function $V_{p}(x)$ in Definition~\ref{D1} satisfies the Foster--Lyapunov equation
\begin{equation}\label{ET2B}
\Ag^u V_p(x) \,\le\, C_0(p) - p\biggl(\frac{\beta}{2m} + \delta\norm{x^-}_1\biggr)
V_{p-1}(x)\quad\forall\,(x,u)\in\RR^m\times\Delta\,,
\end{equation}
for some positive constant $C_0(p)$ depending only on $p$.

\item[\upshape{(}b\upshape{)}]
Under any $v\in\tUsm$, the process
$\process{X}$ in \eqref{E-sde} admits a unique invariant probability measure
$\overline\uppi_v\in\cP(\Rm)$.

\item[\upshape{(}c\upshape{)}]
There exists a constant $C_1(\epsilon)$ depending only
on $\epsilon\in(0,\alpha)$, such that,
under any $v\in\tUsm$, the process
$\process{X}$ in \eqref{E-sde} satisfies 
\begin{equation}\label{ET2C}
\bnorm{P^v_t(x,\,\cdot\,)
-\overline\uppi_v(\,\cdot\,)}_\tv\,\le\,
C_1(\epsilon) (t\vee1)^{1+\epsilon-\alpha}\abs{x}^{\alpha-\epsilon}
\quad\forall\,x\in\Rm\,.
\end{equation}
\end{enumerate}
\end{theorem}

\begin{proof}
Note that, since $\alpha>1$, Lemma~\ref{L3} implies that
$\frac{\fI_\alpha V_p(x)}{1+\abs{V_{p-1}(x)}}$ vanishes
at infinity. Using $\delta$ as in \eqref{ET2A}, it is clear that
$\delta\beta + \delta\frac{m}{2}\le \frac{\beta}{2m}$.
Thus, \eqref{ET2B} is a direct consequence of  Lemmas~\ref{L1} and \ref{L3} together
with the definition in \eqref{E-Ag}.

Clearly, \eqref{ET2B} implies that
\begin{equation}\label{PT2A}
\Ag^v V_p(x) \,\le\, C_0(p) - p\frac{\beta}{2m} V_{p-1}(x)\quad\forall\,x\in\RR^m\,,
\end{equation}
and for any $v\in\tUsm$.
It is well known that the existence of an invariant probability measure
$\overline\uppi_v$ follows from the $C_b$-Feller property and \eqref{PT2A},
while the open-set irreducibility asserted in Theorem~\ref{T1} implies its uniqueness.

Equation \eqref{ET2C} is a direct result of \eqref{PT2A}, Theorem~\ref{T1}
and \cite[Theorem~3.2]{Douc-2009}.
This completes the proof.
\qed\end{proof}

\begin{theorem}\label{T3}
Assume that $\Gamma>0$ and $p \in [1,\alpha)$.
Then, there exists a positive constant $\delta$ such that
\begin{equation*}
\Ag^u V_p(x) \,\le\, \Tilde\kappa_0 - \Tilde\kappa_1 V_p(x)
\quad\forall\,(x,u)\in\RR^m\times\Delta\,.
\end{equation*}
for some positive constants $\Tilde\kappa_0$ and $\Tilde\kappa_1$.
Moreover, under any $v\in\tUsm$,
the process $\process{X}$ admits a unique invariant probability measure
$\overline\uppi_v\in\cP(\Rm)$, and for any $\gamma\in(0, \Tilde\kappa_1)$
there exists
a positive constant $C_\gamma$ such that
\begin{equation*}
\bnorm{P^v_t(x,\,\cdot\,)-\overline\uppi_v(\,\cdot\,)}_{V_{p}}\,\le\,
C_\gamma V_{p}(x)\, \E^{-\gamma t}\,,\qquad x\in\Rm\,,\ t\ge0\,.
\end{equation*}
\end{theorem}

\begin{remark}\label{R2}
We limited our attention to controls in $\tUsm$  only to take advantage of
Theorem~\ref{T1}.
However, if under some $v\in\Usm$ the SDE in \eqref{E-sde} has a unique weak solution
which is an open-set irreducible and aperiodic $C_b$-Feller process,
then it has a unique invariant probability
measure $\overline\uppi_v$, and
the conclusions of Theorems~\ref{T2} and \ref{T3} follow.
\end{remark}

Concerning the lower bound on the rate of convergence, we need not restrict
the controls in $\tUsm$.
The lack of integrability of functions that have strict polynomial growth
of order $\alpha$ (or higher) under the L\'evy measure of $\fI_\alpha$,
plays a crucial role in determining this lower bound.
Consider a $v\in\Usm$ as in Remark~\ref{R2}, and suppose that $\beta>0$.

Then it is shown in Lemma~5.7\,(b) of \cite{APS19} that
\begin{equation}\label{EL5.7b}
\int_\Rm \bigl( \langle e,M^{-1} x\rangle^+\bigr)^p \,\overline\uppi_v(\D{x})
\,<\,\infty \quad\text{for\ some\ } p>0 \quad\Longrightarrow\quad p<\alpha-1\,.
\end{equation}
We use this property in the proof of Theorem~\ref{T4} which follows.
To simplify the notation,
for a function $f$ which is integrable under $\overline\uppi_v$, 
we let $\overline\uppi_v(f) \df \int_\Rm f(x) \overline\uppi_v(\D{x})$.

\begin{theorem}\label{T4}
We assume $\beta>0$.
Suppose that under some $v\in\Usm$ such that
$\Gamma v=0$ a.e.\ the SDE in \eqref{E-sde} has a unique weak solution
which is an open-set irreducible and aperiodic $C_b$-Feller process.
Then the process $\process{X}$ is polynomially ergodic.
In particular, there exists a positive constant $C_2$ not depending
on $v$,
such that for all $\epsilon>0$ we have
\begin{equation*}
\bnorm{P^v_t(x,\,\cdot\,)-\overline\uppi_v(\,\cdot\,)}_\tv
\,\ge\,
C_2 \Bigl(\frac{t\vee 1}{\epsilon}
+ \abs{x}^{\alpha-\epsilon}\Bigr)^{\frac{1-\alpha}{1-\epsilon}}
\quad\forall (t,x)\in\RR_+\times\Rm\,.
\end{equation*}
\end{theorem}

\begin{proof}
The proof uses \cite[Theorem~5.1]{Hairer-16} and some results from \cite{APS19}.
Recall the function $\psi$, and define
\begin{equation*}
\Breve\chi(t)\,\df\, 1+\psi(t)\,,\quad\text{and}\quad
\chi(t) \,\df\, - \Breve\chi(-t)\,.
\end{equation*}
Also, we scale $\chi(t)$ using $\chi_R(t)\df R+\chi(t-R)$, $R\in\RR$.
Thus, $\chi_R(t)= t$ for $t\le R-1$ and
$\chi_R(t) = R-\frac{1}{2}$ for $t\ge R$.

Let
\begin{equation*}
 F(x) \,\df\, \Breve\chi\bigl(\langle e,M^{-1}x\rangle\bigr)\,,\quad
\text{and\ \ } F_{\kappa,R}(x) \,\df\, \chi_R\circ F^\kappa(x)\,,\quad
x\in\Rm\,,\ R>0\,,
\end{equation*}
where $F^\kappa(x)$ denotes the $\kappa^{\mathrm{th}}$ power of $F(x)$,
with $\kappa>0$.

Using the same notation as in \cite[Theorem~5.1]{Hairer-16} whenever possible,
we define $G(x) \df F^{\alpha-\epsilon}(x)$, for $\epsilon\in(0,\alpha-1)$.
Then $\overline\uppi_v(F^{\alpha-\epsilon})=\infty$ by \eqref{EL5.7b}.
Applying the It\^o formula to \eqref{PT2A} we obtain
\begin{equation*}
\Exp_x^v\bigl[V_{\alpha-\epsilon}\bigl(X_t\bigr)\bigr]
- V_{\alpha-\epsilon}(x) \,\le\, C_0(\alpha-\epsilon)\, t\,,\qquad x\in\Rm\,.
\end{equation*}
Since $F^{\alpha-\epsilon} \le \overline{C}_0 V_{\alpha-\epsilon}$
for some constant $\overline{C}_0\ge1$,
the preceding inequality implies that
\begin{equation*}
\Exp_x^v\bigl[F^{\alpha-\epsilon}\bigl(X_t\bigr)\bigr]\,\le\,
\overline{C}_0\bigl(C_0(\alpha-\epsilon) t+ V_{\alpha-\epsilon}(x)\bigr)
\,=:\, g(x,t)\,.
\end{equation*}

Next, we compute a suitable lower bound $f(t)$ for
$\overline\uppi_v\bigl(\{x\colon G(x)\ge t\}\bigr)$.
We have
\begin{equation}\label{PT4B}
\begin{aligned}
\Ag^v F_{1,R}(x) &\,=\, \fI_\alpha F_{1,R}(x) +
\chi_R'\bigl(F(x)\bigr)
\bigl\langle b_v(x),\nabla F(x)\bigr\rangle\\
&\,=\, \fI_\alpha F_{1,R}(x) +
\chi_R'\bigl(F(x)\bigr) \Breve\chi'\bigl(\langle e,M^{-1}x\rangle\bigr)
\bigl(-\beta + \langle e,x\rangle^-\bigr)\,.
\end{aligned}
\end{equation}
Integrating \eqref{PT4B} with respect to $\overline\uppi_v$,
and replacing the variable $R$ with $t$, we obtain
\begin{equation}\label{PT4C}
\beta\,\overline\uppi_v\bigl(\chi_t'(F)h\bigr)\,=\,
\overline\uppi_v\bigl(\fI_\alpha F_{1,t}\bigr)
+ \overline\uppi_v\bigl(\chi_t'(F)\Tilde{h}\bigr)\,,
\end{equation}
where
\begin{equation*}
h(x) \,\df\,\Breve\chi'\bigl(\langle e,M^{-1}x\rangle\bigr)\,,
\quad\text{and}\quad
\Tilde{h}(x)\,\df\, h(x)\, \langle e,x\rangle^- \,.
\end{equation*}
Taking limits as $t\to\infty$ in \eqref{PT4C}, we obtain
\begin{equation}\label{PT4D}
\beta \overline\uppi_v(h) \,=\, \overline\uppi_v(\fI_\alpha F) + 
\overline\uppi_v(\Tilde h)\,.
\end{equation}
Subtracting \eqref{PT4C} from \eqref{PT4D}, gives
\begin{equation}\label{PT4E}
\beta\, \overline\uppi_v(h-\chi_t'(F)h)
\,=\, \overline\uppi_v\bigl(\fI_\alpha (F-F_{1,t})\bigr)
+ \overline\uppi_v\bigl(\Tilde{h}-\chi_t'(F)\Tilde{h}\bigr)\,.
\end{equation}
Note that all the terms in this equation are nonnegative.
Moreover, $\fI_\alpha (F-F_{1,t})(x)$ is nonnegative by convexity,
and thus
\begin{equation}\label{PT4F}
\begin{aligned}
\overline\uppi_v \bigl(\fI_\alpha (F-F_{1,t})\bigr)
&\,\ge\,\inf_{x\in\sB}\,\bigl(\fI_\alpha (F-F_{1,t})(x)\bigr)
\,\overline\uppi_v(\sB)\\
&\,\ge\,\fI_\alpha (F-F_{1,t})(0)\,\overline\uppi_v(\sB)\,.
\end{aligned}
\end{equation}
It is straightforward to show that
$\fI_\alpha (F-F_{1,t})(0)\ge\Hat\kappa t^{1-\alpha}$ for some
positive constant $\Hat\kappa$.
Therefore, by \eqref{PT4E}--\eqref{PT4F}
and the  definition of the functions $F$, $F_{1,R}$ and $h$, 
we obtain
\begin{equation}\label{PT4G}
\begin{aligned}
\overline\uppi_v\bigl(\{x\colon \langle e,M^{-1}x\rangle > t\}\bigr)
&\,\ge\,\overline\uppi_v(h-\chi_t'(F)h)\\
&\,\ge\,\beta^{-1}\,\overline\uppi_v(\sB)
\fI_\alpha (F-F_{1,t})(0)\\
&\,\ge\, \Hat\kappa\, t^{1-\alpha}\,.
\end{aligned}
\end{equation}
Therefore, by \eqref{PT4G}, we have
\begin{align*}
\overline\uppi_v\bigl(\{x\colon G(x)\ge t\}\bigr)&\,=\,
\overline\uppi_v\bigl(\{x\colon
\bigl(\langle e,M^{-1}x\rangle\bigr)^{\alpha-\epsilon} > t\}\bigr)
\nonumber\\
&\,=\, \overline\uppi_v\bigl(\{x\colon \langle e,M^{-1}x\rangle
> t^{\frac{1}{\alpha-\epsilon}}\}\bigr)  \nonumber\\
&\,\ge\, \Hat\kappa\, t^{\frac{1-\alpha}{\alpha-\epsilon}}\,=:\,f(t)\,.
\end{align*}
Next we solve $y f(y) = 2g(x,t)$ for $y=y(t)$, and this gives us
$y=\bigl(\Hat\kappa^{-1}2g(x,t)\bigr)^{\frac{\alpha-\epsilon}{1-\epsilon}}$, and
\begin{equation*}
f(y)\,=\, \Hat\kappa\bigl(\Hat\kappa^{-1}2g(x,t)\bigr)^{\frac{1-\alpha}{1-\epsilon}}
\,=\,
\overline{C}_1 \bigl(C_0(\alpha-\epsilon) t
+ V_{\alpha-\epsilon}(x)\bigr)^{\frac{1-\alpha}{1-\epsilon}}\,,
\end{equation*}
with
\begin{equation*}
\overline{C}_1\,\df\,  \bigl(2\overline{C}_0\bigr)^{\frac{1-\alpha}{1-\epsilon}}
\Hat\kappa^{\frac{\alpha-\epsilon}{1-\epsilon}}\,.
\end{equation*}
Therefore, by \cite[Theorem~5.1]{Hairer-16}, and since $\epsilon$ is arbitrary, we have
\begin{equation}\label{PT4H}
\begin{aligned}
\bnorm{P^v_t(x,\,\cdot\,) - \overline\uppi_v(\,\cdot\,)}_\tv
&\,\ge\, f(y) -\frac{g(x,t)}{y}\\
&\,=\,\frac{\overline{C}_1}{2} \bigl(C_0(\alpha-\epsilon) t
+ V_{\alpha-\epsilon}(x)\bigr)^{\frac{1-\alpha}{1-\epsilon}}
\end{aligned}
\end{equation}
for all $t\ge0$ and $\epsilon\in(0,\alpha-1)$.

As shown in the proof of \cite[Theorem~3.4]{APS19},
there exists a positive constant $\kappa_0'$, not depending on $\epsilon$, such that
\begin{equation}\label{PT4I}
C_0(\alpha-\epsilon)\,\ge\, \kappa_0'(1+\epsilon^{-1})\,.
\end{equation}
Thus the result follows by \eqref{PT4H}--\eqref{PT4I}.
\qed\end{proof}

\section{Ergodic properties of the limiting SDEs arising from queueing models
with service interruptions}\label{S4}

The limiting equations of
multiclass $G/M/n+M$ queues with
asymptotically negligible service interruptions under the $\sqrt{n}$-scaling
in the Halfin--Whitt regime are L\'evy--driven SDEs of the form
\begin{equation}\label{E-sde2}
\D X_t \,=\, b(X_t,U_t)\,\D{t} + \upsigma\D W_t + \D L_t,\qquad X_0=x\in\Rm\,.
\end{equation}
Here, the drift $b$ is as in Section~\ref{S2},
$\upsigma$ is a nonsingular diagonal matrix,
and $\process{L}$ is a compound Poisson process, with a drift $\vartheta$,
and a finite
L\'evy measure $\eta(\D{y})$ which is
supported on a half-line
of the form $\{tw\,\colon t\in[0,\infty)\}$,
with $\langle e,M^{-1} w\rangle>0$. 
This can be established as in Theorem~\ref{T6} in Section~\ref{S5}, assuming that
the control is of the form $U_t = v(X_t)$ for a map $v\colon\cK_+\to\Delta$,
such that $b_v(x)$ is locally Lipschitz,
when the scaling is of order $\sqrt{n}$
(see also Section~4.2 of \cite{APS19}).

As we explain later, under any stationary Markov control, 
the SDE in \eqref{E-sde2} has a unique strong
solution which is an open-set irreducible and aperiodic strong Feller process.
Therefore, as far as the study of the process
$\process{X}$ is concerned,  we do not need to impose a local Lipschitz
continuity condition on the drift, but can allow the control to be any element of $\Usm$.

There are two important parameters to consider.
The first is the parameter $\theta_c$, which is defined by
\begin{equation*}
\theta_c\,\df\,\sup\,\bigl\{\theta\in\Theta_c\bigl\}
\,,\quad\text{with}\quad
\Theta_c\,\df\,\Bigl\{\theta>0\, \colon
\int_{\sB^c}\abs{y}^{\theta}\eta(\D{y})\,<\,\infty\Bigl\}\,.
\end{equation*}
The second is the \emph{effective spare capacity},  defined as 
\begin{equation*}
\widetilde\beta\,\df\,  -\bigl\langle e,M^{-1}\Tilde\ell\,\bigr\rangle\,,
\end{equation*}
where
\begin{equation*}
\Tilde{\ell} \,\df\, \begin{cases}
\ell+\vartheta + \int_{\sB^c}y\,\eta(\D{y})\,,&\text{if\ \ }
\int_{\sB^c}\abs{y}\eta(\D{y})<\infty\\[5pt]
\ell+\vartheta,& \text{otherwise}.
\end{cases}
\end{equation*}

Suppose that $v\in\Usm$ is such that $\Gamma v(x)=0$ a.e. $x$ in $\Rm$.
Then as shown in Lemma~5.7 of \cite{APS19},  the process $\process{X}$ controlled
by $v$ cannot have an invariant probability measure $\overline\uppi_v$
unless $1\in\Theta_c$ and $\widetilde\beta>0$, and moreover,
\begin{equation*}
\int_\Rm \bigl( \langle e,M^{-1} x\rangle^+\bigr)^p \,\overline\uppi_v(\D{x})
\,<\,\infty \quad\text{for\ some\ } p>0 \quad\Longrightarrow\quad p+1\in\Theta_c\,.
\end{equation*}
In addition,
$\widetilde\beta=\int_\Rm \langle e,x\rangle^-\,\overline\uppi_v(\D{x})$
 \cite[Theorem~3.4\,(b)]{APS19}.
Conversely, $1\in\Theta_c$
and $\widetilde\beta>0$ are  sufficient  for $\process{X}$ to have
an invariant probability measure $\overline\uppi_v$
under any constant control $v$,
and $\overline\uppi_v\in\cP_p(\Rm)$ if $p+1\in\Theta_c$
(see Theorems~3.2 and 3.4\,(b) in \cite{APS19}).

On the other hand, if $\Gamma>0$, that is, it has positive diagonal elements,
then $\process{X}$ is geometrically ergodic under any constant Markov control,
and $\overline\uppi_v\in\cP_{\theta}(\Rm)$ for any $\theta\in\Theta_c$
\cite[Theorem~3.5]{APS19}.
This bound is tight since, in general, if under some Markov control $v$
the process
$\process{X}$ has an invariant probability
measure $\overline\uppi_v\in\cP_p(\Rm)$, then necessarily $p\in\Theta_c$.

We extend the results derived for constant Markov controls in \cite{APS19}
to all controls in $\Usm$.
Recall the definition in \eqref{frakd}.
Let
\begin{equation*}
\Tilde{b}(x,u) \,\df\, b(x,u) + \Tilde\ell - \ell\,,
\end{equation*}
and $\Tilde{b}_v(x)=\Tilde{b}\bigl(x,v(x)\bigr)$ for $v\in\Usm$. 
As explained in Section~\ref{S2}, we assume, without loss of generality,
that the constant term in $\Tilde{b}$ is as in \eqref{E-drift2} with
$\beta$ replaced by $\widetilde\beta$.

We define the operator $\mathcal{A}^u$ on $\Cc^2$ functions by
\begin{equation*}
\cA^u f(x) \,\df\, \Lg^u f(x) + \mathfrak{J}_{\eta}f(x)\,,\quad
(x,u)\in\Rm\times\Delta\,,
\end{equation*}
where
\begin{equation}\label{E-Lg}
\Lg^u f(x) \,=\, \frac{1}{2}\trace\bigl(\sigma\sigma'\nabla^2f(x)\bigr)
+ \bigl\langle\Tilde{b}(x,u),\nabla f(x)\bigr\rangle\,,\quad (x,u)\in\Rm\times\Delta\,,
\end{equation}
and
\begin{equation*}
\mathfrak{J}_{\eta}f(x)\,\df\,\int_\Rm \mathfrak{d} f(x;y)\,\eta(\D{y})\,,
\qquad x\in\Rm\,.
\end{equation*}
Also, $\Lg^v$ is defined as in \eqref{E-Lg} by replacing $u$ with
$v(x)$ for a control $v\in\Usm$, and analogously for $\cA^v$. 

It follows from the results in \cite{Gyongy-96} that, for any
$v\in\Usm$, the diffusion
\begin{equation}\label{E-sde-aux}
\D{\Tilde X_t} \,=\,
\Tilde{b}(\Tilde X_t,v(\Tilde X_t))\,\D{t} + \upsigma(\Tilde X_t)\,\D{W_t}\,,
\quad \Tilde X_0 = x\in \RR^d
\end{equation}
has a unique strong solution.
Also, as shown in \cite{Skorokhod-89}, since the
the L\'evy measure is finite,
the solution of \eqref{E-sde2} can be constructed in a piecewise fashion
using the solution of \eqref{E-sde-aux} (see also \cite{Li-05}).
It thus follows that, under any stationary Markov control, \eqref{E-sde2} has a unique
strong solution which is a strong Markov process.
In addition, its transition probability  $P^v_t(x,dy)$ satisfies 
$P^v_t(x,B)>0$ for all $t>0$, $x\in\Rm$ and $B\in\cB(\Rm)$
with positive Lebesgue measure. Thus, under any $v\in\Usm$,
the process $\process{X}$ is open--set irreducible and aperiodic.

Recall Definition~\ref{D1}.
In order to handle the second order derivatives in $\cA^u$
we need to scale the Lyapunov function $V_p$.
This is done as follows.
With $\psi$ as in Definition~\ref{D1}, we define
\begin{equation*}
\psi_\delta(t) \,\df\, \psi (\delta t)\,,\quad\text{and}\quad
\Psi_\delta(x) \,\df\, \sum_{i\in\cI} \frac{\psi_\delta(x_i)}{\mu_i}\,,
\quad \delta\in(0,1]\,,
\end{equation*}
and let
\begin{equation*}
\Lyap_{p,\delta}(x) \,\df\, \biggl(\delta^2\Psi(-x)+ \Psi_\delta(x)
+ \frac{m}{\min_{i\in\cI}\mu_i}\biggr)^p\,.
\end{equation*}
Note that $\Lyap_{1,\delta}$ is bounded away from $0$ uniformly
in $\delta\in(0,1]$.
Here we use the inequality
$\sum_{i\in\cI} \psi_\delta'(x_i) x_i \ge \delta\norm{x^+}^{}_1 - \frac{m}{2}$.
Then, under the assumption that $\widetilde\beta>0$, the drift inequalities
take the form
\begin{align}\label{E-dr}
&\bigl\langle \Tilde{b}(x,u), \nabla \Lyap_{p,\delta} (x)\bigr\rangle \nonumber\\[5pt]
&\mspace{50mu}\,\le\,
\begin{cases}
p\delta\Bigl(\delta\widetilde\beta +
\frac{m}{2\delta}(1+\delta^2) - \delta \norm{x}^{}_1\Bigr) \Lyap_{p-1,\delta}(x)
&\forall\,x\in\cK_-\,,\\[5pt]
- p\delta\bigl(\frac{\widetilde\beta}{m}-\delta\widetilde\beta-\delta\frac{m}{2}
+\delta\norm{x^-}{}_1\bigr)\,\Lyap_{p-1,\delta}(x) 
&\forall\,(x,u)\in\cK_+\times\Delta\,.
\end{cases}
\end{align}

The following result is analogous to Theorem~\ref{T2}.

\begin{theorem}\label{T5}
Assume $\widetilde\beta>0$, and $1\in\Theta_c$.
Let $p\in\Theta_c$ with $p>1$.
Then the following hold.
\begin{enumerate}
\item[\upshape{(}a\upshape{)}]
There exists $\delta>0$, a positive constant
$\widetilde{C}_0$, and a compact set $K$ such that
\begin{equation}\label{ET5A}
\cA^u \Lyap_{p,\delta}(x) \,\le\,
\widetilde{C}_0\Ind_K(x) - p\delta\frac{\widetilde\beta}{2m}
\Lyap_{p-1,\delta}(x)\quad\forall\,(x,u)\in\RR^m\times\Delta\,.
\end{equation}

\item[\upshape{(}b\upshape{)}]
Under any $v\in\Usm$, the process
$\process{X}$ in \eqref{E-sde} admits a unique invariant probability measure
$\overline\uppi_v\in\cP(\Rm)$.

\item[\upshape{(}c\upshape{)}]
For any $\theta\in\Theta_c$ there exists a constant
 $\widetilde{C}_1(\theta)$ depending only
on $\theta$, such that,
under any $v\in\Usm$, the process
$\process{X}$ in \eqref{E-sde} satisfies 
\begin{equation*}
\bnorm{P^v_t(x,\,\cdot\,)-\overline\uppi_v(\,\cdot\,)}_\tv\,\le\,
\widetilde{C}_1(\theta) (t\vee1)^{1-\theta}\abs{x}^{\theta}
\quad\forall\,x\in\Rm\,.
\end{equation*}
\end{enumerate}
\end{theorem}

\begin{proof}
It is straightforward to show that $\psi''_\delta(t)\le 2\delta^2$
and $\psi'_\delta(t) \le \delta$ for all $t\in\RR$.
An easy calculation then shows that
there exists a positive constant $C$ such that
\begin{equation}\label{PT5A}
\trace\bigl(\sigma\sigma'\nabla^2 \Lyap_{p,\delta}(x)\bigr)
\,\le\,Cp^2\delta^2\bigl( \Lyap_{p-1,\delta}(x)+\Lyap_{p-2,\delta}(x)\bigr)
\end{equation}
for all $p\ge1$ and $x\in\Rm$.
Recall that $\Lyap_{1,\delta}$ is bounded away from $0$ uniformly
in $\delta\in(0,1]$.
This of course implies that $\Lyap_{p-2,\delta}$ is bounded by
some fixed multiple of $\Lyap_{p-1,\delta}$ for all $p\ge1$.
Therefore, \eqref{E-dr} and \eqref{PT5A} imply that for some
small enough positive $\delta$ we can chose
a positive constant
$\widetilde{C}_0'$, and a compact set $K'$ such that
\begin{equation}\label{PT5B}
\Lg^u \Lyap_{p,\delta}(x) \,\le\,
\widetilde{C}_0'\Ind_{K'}(x) - p\delta\frac{3\widetilde\beta}{4m}
\Lyap_{p-1,\delta}(x)\quad\forall\,(x,u)\in\RR^m\times\Delta\,.
\end{equation}

If $p\in\Theta_c$, then \cite[Lemma~5.1]{APS19} asserts that
$\mathfrak{J}_{\eta}\Lyap_{p,\delta}$ vanishes at infinity for
$p<2$, and
$\mathfrak{J}_{\eta}\Lyap_{p,\delta}$ is of order
$\abs{x}^{p-2}$ for $p\ge2$.
This together with \eqref{PT5B} implies \eqref{ET5A}.
The rest are as in the proof of Theorem~\ref{T2}.
\qed\end{proof}

If $\Gamma>0$, then the arguments in the proof of Theorem~\ref{T5}
together with Lemma~\ref{L2} show that the process $\process{X}$ is
geometrically ergodic uniformly over $v\in\Usm$.
Thus we obtain the analogous results to Theorem~\ref{T3}.
We omit the details which are routine.

Note that the assumption that the L\'evy measure $\eta(\D{y})$ is
supported on a half-line
of the form $\{tw\,\colon t\in[0,\infty)\}$,
with $\langle e,M^{-1} w\rangle>0$ has not been used,
and is not needed in Theorem~\ref{T5}.
Under this assumption we can obtain a lower bound of the rate
of convergence analogous to equation (3.9) in \cite{APS19},
by mimicking the arguments in that paper.
We leave the details to the reader.

\begin{remark}
With heavy-tailed arrivals and 
asymptotically negligible service interruptions under the common
$n^{\nicefrac{1}{\alpha}}$-scaling for $\alpha \in (1,2)$,
in the modified Halfin--Whitt regime, the limit process is an SDE driven by
an anisotropic $\alpha$-stable process (with independent $\alpha$-stable components)
as in \eqref{E-sde}, and a compound Poisson process with a finite
L{\'e}vy measure as in \eqref{E-sde2}.
This can be established as in Theorem~\ref{T6}, under the same scaling
assumptions in Section~4.2 of \cite{APS19}.
Thus the generator 
is given by
\begin{equation*}
\Hat\Ag^u f(x) \,\df\,
\bigl\langle\Tilde{b}(x,u),\nabla f(x)\bigr\rangle
+ \mathfrak{J}_{\eta}f(x) + \fI_\alpha f(x)\,,
\end{equation*}
and $\Hat\Ag^v$ is defined analogously by replacing $u$ with $v(x)$ 
for $v\in\tUsm$.

To study this equation,
 we use the Lyapunov function $V_p$ in Definition~\ref{D1},
with $p\in[1,\alpha)\cap\Theta_c$.
Following the proof of Theorem~\ref{T5}, and also using Lemma~\ref{L3},
it follows that there exists $\delta>0$ sufficiently small,
a constant $\widehat{C}_0$ and a compact set $\widehat{K}$
such that
\begin{equation*}
\Hat\Ag^u V_p(x) \,\le\,
\widehat{C}_0\Ind_{\widehat{K}}(x) - p\frac{\widetilde\beta}{2m}
V_{p-1}(x)\quad\forall\,(x,u)\in\RR^m\times\Delta\,.
\end{equation*}
Thus, \eqref{ET2C} holds for any $\epsilon$ such that
$\alpha-\epsilon\in\Theta_c$.
The results of Theorem~\ref{T3} also follow provided we select
$p\in[1,\alpha)\cap\Theta_c$.
However the lower bound is not necessarily the one in Theorem~\ref{T4}.
Instead we can obtain a lower bound in the form of equation (3.9) in \cite{APS19}.
\end{remark}

\section{Multiclass $G/M/n+M$ queues with heavy-tailed arrivals}\label{S5}

As in \cite[Subsection~4.1]{APS19}, consider $G/M/n+M$ queues with $m$ classes of
customers and one server pool of $n$ parallel servers.
Customers of each class form their own queue and are served in the first-come
first-served (FCFS) service discipline.
Customers of different classes are scheduled to receive service under the work
conserving constraint, that is, non-idling whenever customers are in queue. 
We assume that the arrival process of each class is  renewal with heavy-tailed
interarrival times. The service and patience times are exponentially distributed
with class-dependent rates. The arrival, service and abandonment processes of each
class are mutually independent. 

We consider a sequence of such queueing models 
indexed by $n$ and let $n\to\infty$. 
Let $A^n_i$, $i=1,\dotsc,m$, be the arrival process of class-$i$ customers
with arrival rate $\lambda^n_i$. Assume that $A^n_i$'s are mutually independent.
Define the FCLT-scaled arrival processes
$\Hat{A}^n= (\Hat{A}^n_1,\dotsc,\Hat{A}^n_m)'$ by
$\Hat{A}^n_i\df n^{-\nicefrac{1}{\alpha}} (A^n_i- \lambda^n_i \varpi)$,
$i=1,\dotsc,m$,
where $\varpi(t)\equiv t$ for each $t\ge 0$, and $\alpha \in (1, 2)$. 
We assume that 
\begin{equation} \label{E4.1}
\nicefrac{\lambda^n_i}{n} \,\to\, \lambda_i\,>\,0, \quad \text{and}\quad
\ell^n_i\,\df\, n^{-\nicefrac{1}{\alpha}}(\lambda^n_i - n \lambda_i)
\,\to\, \ell_i\in \RR\,,
\end{equation}
for each $i=1,\dotsc,m$, as $n\to\infty$, and
that the arrival processes satisfy an FCLT
\begin{equation*}
\Hat{A}^n \,\Rightarrow\, \Hat{A} = (\Hat{A}_1,\dotsc, \Hat{A}_m)'
\qquad\text{in\ } (D_m,\, M_1), \ \text{as\ } n\to\infty\,,
\end{equation*}
where the limit processes $\Hat{A}_i$, $i=1,\dotsc,m$, are mutually
independent symmetric $\alpha$-stable processes with $\Hat{A}_i(0)\equiv 0$,
and $\Rightarrow$ denotes
weak convergence and $(D_m,M_1)$ is the space of $\RR^m$-valued c{\`a}dl{\`a}g
functions endowed with the product $M_1$ topology \cite{WW02}.  
The processes $\Hat{A}_i$ have the same stability parameter $\alpha$,
with possibly different ``scale'' parameters $\xi_i$. 
Note that if the arrival process of each class is renewal with regularly
varying interarrival times of parameter $\alpha$, then we obtain the
above limit process.
Let $\mu_i$ and $\gamma_i$ be the service and abandonment rates for class-$i$
customers, respectively. 

\noindent \emph{The modified Halfin-Whitt regime}.
The parameters satisfy
\begin{equation*}
n^{1-\nicefrac{1}{\alpha}} (1-\rho^n)\,\xrightarrow[n\to\infty]{}\, \rho
\,=\, - \sum_{i=1}^m \frac{\ell_i}{\mu_i}\,,
\end{equation*}
where $\rho^n \df \sum_{i=1}^m \frac{\lambda^n_i}{n \mu_i}$
is the aggregate traffic intensity. This follows from \eqref{E4.1}. 
Let $\rho_i\df \nicefrac{\lambda_i}{\mu_i}$ for $i\in \cI$.

Let $X^n=(X^n_1,\dotsc,X^n_d)'$, $Q^n= (Q^n_1,\dotsc,Q^n_d)'$, and
$Z^n= (Z^n_1,\dotsc,Z^n_d)'$ be the processes counting the number of customers
of each class in the system, in queue, and in service, respectively. 
We consider work-conserving scheduling policies that are non-anticipative and
allow preemption (namely, service of a customer can be interrupted at any time
to serve some other class of customers and will be resumed at a later time). 
Scheduling policies determine the allocation of service capacity, i.e.,
the $Z^n$ process, which must satisfy the condition that
$\langle e, Z^n\rangle = \langle e,X^n\rangle \wedge n$ at each time,
as well as the balance equations $X^n_i = Q^n_i + Z^n_i$ for each $i$.

Define the FCLT-scaled processes $\Hat{X}^n=(\Hat{X}^n_1,\dotsc,\Hat{X}^n_d)'$,
$\Hat{Q}^n= (\Hat{Q}^n_1,\dotsc,\Hat{Q}^n_d)'$, and
$\Hat{Z}^n= (\Hat{Z}^n_1,\dotsc,\Hat{Z}^n_d)'$ by
\begin{equation*}
\Hat{X}_i^n \,\df\, n^{-\nicefrac{1}{\alpha}} (X^n_i - \rho_i n)\,,
\qquad \Hat{Q}_i^n \,\df\, n^{-\nicefrac{1}{\alpha}} Q^n_i\,, \qquad \Hat{Z}_i^n
\,\df\, n^{-\nicefrac{1}{\alpha}} (Z_i^n - \rho_i n) \,.
\end{equation*}

We need the following extension of
Theorem~1.1 in \cite{PW10}.
Let $\phi\colon D([0,T],\RR^m)\to D([0,T],\RR^m)$ denote
the mapping $x\mapsto y$ defined by the integral representation
\begin{equation*}
y(t) \,=\, x(t) + \int_0^t h(y(s))\, \D s\,, \quad  t\ge 0\,.
\end{equation*}
It is shown in \cite[Theorem~1.1]{PW10} that the mapping $\phi$ is continuous
in the Skorohod $M_1$ topology when $m=1$ and the function $h$ is Lipschitz continuous. 
The lemma which follows extends this result to functions
$h\colon \RR^m \to \RR^m$ which are locally Lipschitz continuous and have
at most linear growth.

\begin{lemma} \label{L4}
Assume that $h$ is locally Lipschitz and has at most linear growth.
Then the mapping $\phi$ defined above is continuous in $(D_m, M_1)$, 
the space $D([0,T],\RR^m)$ endowed with the product $M_1$ topology. 
\end{lemma}

\begin{proof}
Assume that $x_n \to x$ in $D_m$ with the product $M_1$ topology as $n\to\infty$.
Let $x^i$ be the $i^{\rm th}$ component of $x$, and similarly for $x^i_n$. 
Let 
\begin{equation*}
G_x \,\df\, \bigl\{(z,t) \in \RR^m\times [0,T]\colon z^i \in [x^i(t-), x^i(t)] \quad
\text{for each\ } i =1,\dotsc,m \bigr\}\,,
\end{equation*}
be the (weak) graph of $x$, and similarly, $G_{x_n}$ for $x_n$;
see Chapter~12.3.1 in \cite{WW02}. 
Then following the proof of Theorem~1.2 in \cite{PW10}, it can be shown that 
there exist parametric representations $(u,r)$ and $(u_n, r_n)$ of $x$ and $x_n$,
that map $[0,1]$ onto the graphs $G_x$ and $G_{x_n}$ of $x$ and $x_n$, respectively,
and satisfy the properties below.
In the construction of the time component as in Lemma 4.3 of \cite{PW10}, 
the discontinuity points of all the $x^i$ components need to be included,
and then the spatial component
can be done similarly as in the proof of that lemma. 
\begin{itemize}
\item
The time (domain) components $r, r_n \in C([0,1], [0,T])$ are nondecreasing
functions satisfying
$r(0) =r_n(0)=0$ and $r_n(1) =r_n(1)=T$, and such that
$r$ and $r_n$ are absolutely continuous with respect to Lebesgue measure on $[0,1]$.
\item
The derivatives $r'$ and $r'_n$ exist for all $n$ and satisfy
 $\norm{r'}_\infty<\infty$, $\sup_n \norm{r'_n}_\infty < \infty$,
and $\norm{r'_n-r'}_{L^1} \to 0$,
where $\norm{r}_\infty \df \sup_{s\in [0,1]}\abs{r(s)}$,
and $\norm{\,\cdot\,}_{L^1}$ denotes the $L^1$ norm.
\item
The spatial components $u= (u^1,\dotsc,u^m)$
and $u_n = (u_n^1,\dotsc,u_n^m)$, $n\in\NN$, lie in $C([0,1],\Rm)$,
and satisfy $u(0) = x(0)$, $u(1) = x(T)$,
$u_n(0) = x_n(0)$, $u_n(1) = x_n(T)$, and
$ \norm{u_n-u}_\infty \to 0$ as $n\to\infty$.
\end{itemize}

As shown in the proof of Theorem~1.1 in \cite{PW10}, there exist  parametric
representations $(u_y,r_y)$ and $(u_{y_n}, r_{y_n})$ of $y$ and $y_n$,
respectively, with $r_y = r$ and $r_{y_n} =r_n$, satisfying
\begin{equation} \label{uy-rep}
u_{y_n}(s) \,=\, u_n(s) + \int_0^s h\bigl(u_{y_n}(w)\bigr) r'_n(w)\,\D w\,,
\quad  s \in [0,1]\,,
\end{equation}
and similarly for $u_{y}(s)$.
Here, $(u,r)$ and $(u_n,r_n)$ are the parametric representations of
$x$ and $x_n$, respectively, whose properties are summarized above.

Since $x_n \to x$ in $(D_m, M_1)$ as $n\to\infty$,
we have $\sup_{n}\, \norm{u_{n}}_\infty< \infty$.
Taking norms in \eqref{uy-rep}, and using also the property
$\sup_n\, \norm{r'_n}_\infty < \infty$, and the
linear growth of $h$, an application of Gronwall's lemma shows that
$\sup_n\,\norm{u_{y_n}}_\infty \le  R $
for some constant $R$.
Enlarging this constant if necessary, we may also assume that
$\norm{u_{y}}_\infty \le  R$.
By the representation in \eqref{uy-rep}, we have 
\begin{align*}
\abs{u_{y_n}(s) - u_y(s)} & \,\le\, \abs{u_n(s) - u(s)} 
+ \babss{\int_0^s \bigl(h(u_{y_n}(w)) - h(u_{y}(w)) \bigr) r'_n(w)\,\D w } 
\nonumber \\
& \mspace{80mu}  +
\babss{\int_0^s h(u_{y}(w))  r'_n(w)\,\D w - \int_0^s h(u_{y}(w))  r'(w)\,\D w}\,. 
\end{align*}
Let $\kappa_{R}$ be a Lipschitz constant of $h$ on the ball $\sB_{R}$.
Then, applying Gronwall's lemma once more, we obtain 
\begin{equation*}
\norm{u_{y_n} - u_y}_\infty \,\le\, \Bigl(\norm{u_n-u}_\infty
+  \norm{r'_n-r'}_{L^1}\,\sup_{\sB_R}\, h\Bigr)\,
\E^{\kappa_{R}\norm{r'_n}_\infty}
\,\xrightarrow[n\to\infty]{}\, 0\,. 
\end{equation*}
This completes the proof.
\qed\end{proof}

\begin{remark}\label{R4}
Suppose $h$, $x$, $x_n$, and $y$ are as in Lemma~\ref{L4}, but
$y_n$ satisfies
\begin{equation*}
y_n(t) \,=\, x_n(t) + \int_0^t h_n(y_n(s))\, \D s\,, \quad  t\ge 0\,,
\end{equation*}
for some sequence $h_n$ which converges to $h$ uniformly on compacta.
Then a slight variation of the proof of Lemma~\ref{L4}, shows
that $y_n \to y$ in $D_m$.
\end{remark}

\noindent \emph{Control approximation.}
Given a continuous map $v\colon\cK_+\to\Delta$, we construct a stationary Markov
control for the $n$-system which approximates it in a suitable manner.

%
%
%

Recall that $\langle e,\rho\rangle=1$. Let
\begin{equation*}
\mathscr{X}_n\,\df\,
\bigl\{  n^{-\nicefrac{1}{\alpha}} (y - \rho n)\colon
y\in\ZZ_+^m\,, \langle e,y\rangle > n\bigr\}\,,
\end{equation*}
and
$\mathfrak{Z}_n= \mathfrak{Z}_n(\Hat{x})$ denote the set of work-conserving actions
at $\Hat{x}\in\mathscr{X}_n$.
It is clear that a work-conserving action $\Hat{z}^n\in\mathfrak{Z}_n(\Hat{x})$
can be parameterized
via a map $\Hat{U}^n\colon\cK_+\to\Delta$, satisfying
\begin{equation} \label{U-para}
\qquad \Hat{z}_i^n(\Hat{x}) \,=\, \Hat{x}_i - 
\langle e,\Hat{x}\rangle^+ \Hat{U}^n_i(\Hat{x})\,.
\end{equation}

Consider the mapping defined in \eqref{U-para}
from $\Hat{z}^n\in\mathfrak{Z}_n(\Hat{x})$ to $\Hat{U}^n$,
and denote its image as $\widehat{\mathscr{U}}_n(\Hat{x})$.
Let
\begin{equation}\label{E-hatU}
\Hat{U}^n[v](\Hat{x})\,\in\, \Argmin_{u\in \widehat{\mathscr{U}}_n(\Hat{x})}\,
\babs{\langle e,\Hat{x}\rangle\,u-
\langle e,\Hat{x}\rangle\,v(\Hat{x})}\,,\quad \Hat{x}\in \mathscr{X}_n\,.
\end{equation}
The function $\Hat{U}^n[v]$ has
the following property. 
There exists a constant $\Check{c}$ such that with
$\Check\sB_n$ denoting the ball of radius
$\Check{c} n^{\Check\alpha}$ in $\Rm$, with $\Check\alpha \df 1-\nicefrac{1}{\alpha}$,
then
\begin{equation}\label{E-approx}
\sup_{\Hat{x}\in\Check\sB_n\cap \mathscr{X}_n}\,
\Babs{\langle e,\Hat{x}\rangle\,\Hat{U}^n[v](\Hat{x})-
\langle e,\Hat{x}\rangle\,v(\Hat{x})}\,\le\,
n^{-\nicefrac{1}{\alpha}}\,.
\end{equation}

We have the following functional limit theorem.

\begin{theorem}\label{T6}
Let $v\in\tUsm$.
Under any stationary Markov control $\Hat{U}^n[v]$ defined in \eqref{E-hatU},
and provided there exists $X(0)$ such that
$\Hat{X}^n(0) \Rightarrow X(0)$ as $n\to\infty$, we have
\begin{equation*}
\Hat{X}^n \,\Rightarrow\, X \qquad\text{in} \quad (D_m, M_1)
\quad \text{as} \quad n \to\infty\,,
\end{equation*}
where the limit process $X$ is the unique strong solution to the SDE in \eqref{E-sde}.
The parameters in the drift are given by
$\ell_i$ in \eqref{E4.1},  $\mu_i$, and $\gamma_i$, for $i=1,\dotsc,m$.
\end{theorem} 

\begin{proof}
The FCLT-scaled processes $\Hat{X}^n_i$, $i=1,\dotsc,m$, can be represented as
\begin{align*}
\Hat{X}^n_i(t) &\,=\, \Hat{X}^n_i(0) + \ell^n_i t
- \mu_i \int_0^t \Hat{Z}^n_i(s)\,\D s - 
\gamma_i \int_0^t \Hat{Q}^n_i(s)\, \D s + \Hat{A}^n_i(t) - \Hat{M}^n_{S,i}(t) - \Hat{M}
^n_{R,i}(t)
\end{align*}
where $\ell^n_i$ is defined in \eqref{E4.1}, with
\begin{equation*}
\Hat{M}^n_{S,i}(t)\,=\, n^{-\nicefrac{1}{\alpha}} \bigg(S_i^n \bigg(\mu_i \int_0^t 
Z^n_i(s)\,\D s \bigg) - \mu_i \int_0^t Z^n_i(s) \D s \bigg) \,,
\end{equation*}
\begin{equation*}
\Hat{M}^n_{R,i}(t) \,=\, n^{-\nicefrac{1}{\alpha}} \bigg(R_i^n \bigg(\gamma_i \int_0^t 
Q^n(s)\,\D s \bigg) - \gamma_i \int_0^t Q^n_i(s) \D s \bigg) \,,
\end{equation*}
and $S^n_i, R^n_i$, $i=1,\dotsc,m$, are mutually independent rate-one Poisson
processes, representing the service and reneging (abandonment), respectively. 

The result can be established by mimicking the arguments in the proof of in
\cite[Theorem~4.1]{APS19}, and applying Lemma~\ref{L4} and Remark~\ref{R4},
using the function
\begin{equation*}
h_n(x)\,\df\ \ell^n + M\bigl(x- \langle e, x\rangle^+\hat{U}^n[v](x)\bigr)
- \langle e, x\rangle^+ \Gamma \hat{U}^n[v](x)\,,
\end{equation*}
and the bound in \eqref{E-approx}.
\qed\end{proof}

\section{Concluding remarks}
We have extended some of the results in \cite{APS19} stated for constant controls,
to stationary Markov controls resulting in a locally Lipschitz drift  in the case
of SDEs driven by $\alpha$-stable processes,
and to all stationary Markov controls in the case of SDEs driven by
a Wiener process and a compound Poisson process.
The results in this paper can also be viewed as an extension of some results
in \cite{AHP18}.
However, the work in \cite{AHP18} also studies the prelimit process and establishes
tightness of the stationary distributions.
To the best of our knowledge,
this is an open problem for systems with arrival processes which are
renewal with heavy-tailed
interarrival times (no second moments).
This problem is very important and worth pursuing.

\begin{acknowledgement}
This research was supported in part by 
the Army Research Office through grant W911NF-17-1-001,
in part by the National Science Foundation through grants DMS-1715210,
CMMI-1538149 and DMS-1715875,
and in part by Office of Naval Research through grant N00014-16-1-2956.
Financial support through Croatian Science Foundation under the project 8958
(for N. Sandri\'c) is gratefully acknowledged.
\end{acknowledgement}

\bibliographystyle{spmpsci}
\bibliography{IMA-paper-Ref}

\end{document}